\documentclass{amsart}

\usepackage[dvipsnames]{xcolor}

\usepackage{hyperref}

\hypersetup{
     colorlinks=true,
     linkcolor=blue!75!black,
     filecolor=blue!75!black,
     citecolor = green!60!black,      
     urlcolor=magenta,
     }

\usepackage{amssymb}
\usepackage{ bbold }
\usepackage{xypic}
\xyoption{all}
\usepackage{url}
\usepackage{mathrsfs}
\usepackage{tikz-cd}
\usepackage{eucal}

\title[Fiberwise building and stratification]{Fiberwise building and stratification in tensor triangular geometry}

\author[J. O. G\'omez]{Juan Omar G\'omez}
\thanks{}
\address{Fakult\"at f\"ur Mathematik,
Universit\"at Bielefeld, D-33501 Bielefeld, Germany.}
\email{jgomez@math.uni-bielefeld.de}

\newcommand{\comments}[1]{}

\newcommand{\Mod}{\operatorname{Mod}\nolimits}
\newcommand{\perm}{\operatorname{perm}\nolimits}
\newcommand{\Perm}{\operatorname{Perm}\nolimits}

\newcommand{\res}{\operatorname{res}\nolimits}

\newcommand{\StMod}{\operatorname{StMod}\nolimits}

\newcommand{\Spc}{\operatorname{Spc}\nolimits}
\newcommand{\Spec}{\operatorname{Spec}\nolimits}

\DeclareMathOperator{\mmod}{mod}
\newcommand{\mmmod}{\mmod\kern-0.1em\text{-}}%          

\def \D{{\mathcal D}}

\def \K{{\mathcal K}}

\def \P{{\mathcal P}}
\def \pp{\mathfrak p }

\def \T{{\mathcal T}}

\makeatletter
\newcommand*{\doublerightarrow}[2]{\mathrel{
  \settowidth{\@tempdima}{$\scriptstyle#1$}
  \settowidth{\@tempdimb}{$\scriptstyle#2$}
  \ifdim\@tempdimb>\@tempdima \@tempdima=\@tempdimb\fi
  \mathop{\vcenter{
    \offinterlineskip\ialign{\hbox to\dimexpr\@tempdima+1em{##}\cr
    \rightarrowfill\cr\noalign{\kern.5ex}
    \rightarrowfill\cr}}}\limits^{\!#1}_{\!#2}}}
\newcommand*{\triplerightarrow}[1]{\mathrel{
  \settowidth{\@tempdima}{$\scriptstyle#1$}
  \mathop{\vcenter{
    \offinterlineskip\ialign{\hbox to\dimexpr\@tempdima+1em{##}\cr
    \rightarrowfill\cr\noalign{\kern.5ex}
    \rightarrowfill\cr\noalign{\kern.5ex}
    \rightarrowfill\cr}}}\limits^{\!#1}}}
\makeatother

\newcommand{\nc}{\newcommand}
\nc{\dmo}{\DeclareMathOperator}
\nc{\Weyl}[2]{{#1}/\!\!/{#2}}%{\overline{#1}_{#2}}%{W_{\!#1}{#2}}% sometimes Weyl is N_G(H)/C_G(H) not mod H.
\nc{\WGH}{\Weyl{G}{H}}
\nc{\WGK}{\Weyl{G}{K}}
\nc{\WGL}{\Weyl{G}{L}}
\nc{\WGN}{\Weyl{G}{N}}

\theoremstyle{plain}

\newtheorem{theorem}{Theorem}[section]
\newtheorem{proposition}[theorem]{Proposition}
\newtheorem{corollary}[theorem]{Corollary}
\newtheorem{lemma}[theorem]{Lemma}

\theoremstyle{definition}

\newtheorem{remark}[theorem]{Remark}
\newtheorem{recollection}[theorem]{Recollection}

\newtheorem{notation}[theorem]{Notation}
\newtheorem{construction}[theorem]{Construction}

\theoremstyle{equation}
\newtheorem{hypotheses}[theorem]{Hypotheses}
\newtheorem{hypothesis}[theorem]{Hypothesis}

\thanks{}
\begin{document}

\maketitle

\begin{abstract}
We establish conditions on a family of coproduct-preserving tt-functors $f_i\colon \T \to \T_i$ between tt-categories with small coproducts, ensuring that the localizing tensor ideal generated by an object $x \in \T$ is determined by those objects whose image under each $f_i$ lies in the localizing tensor ideal generated by $f_i(x)$ for all $i$. This leads to a fiberwise criterion for stratification in the setting of rigidly-compactly generated tt-categories.
As an application, we prove that the big derived category of permutation modules for a finite group over an arbitrary Noetherian base is stratified. Moreover, our methods extend to the category of representations of a finite group scheme over a Noetherian base, thereby recovering a recent result from the literature.
\end{abstract}

\setcounter{tocdepth}{1}

\tableofcontents

%------------------------------------------------------------------------------------------------

\section{Introduction}

Recent results on the geometry of certain tensor-triangulated categories arising in algebraic geometry and representation theory point toward a \textit{fiberwise phenomenon} in tensor-triangular geometry. For example, Lau \cite{Lau23} proved that the Balmer spectrum of the category of perfect complexes $\mathrm{Perf}(\mathcal{X})$, where $\mathcal{X}$ is the quotient stack of an affine scheme by the action of a finite group, is completely controlled by fiber functors arising from suitable base changes.

On the other hand, in \cite{BBIKP}, the authors introduced a big category of representations $\mathrm{Rep}(G, R)$ for a finite group scheme $G$ over an arbitrary Noetherian base ring $R$, and showed that the geometry of $\mathrm{Rep}(G, R)$ is entirely determined by the fiber functors obtained via base change to the residue fields of $R$. In fact, this result extends Lau’s theorem from the ``small'' to the ``big'' setting, at least in the case where the action of the group on the affine scheme is trivial.

With this in mind, we focus on the question: for a given family of functors $\{f_i\colon \T\to \T_i\}_{i\in I}$, what conditions on $\T$, $\T_i$, and the functors $f_i$ allow us to transport geometric information from the $\T_i$ to geometric information of $\T$?

This question was previously investigated in \cite{BCHS23}, where the authors provide conditions ensuring that the induced map on Balmer spectra is surjective. Our approach differs in that we seek to understand the localizing tensor ideals of $\T$ in terms of those of the categories $\T_i$, rather than focusing on the thick tensor ideals of $\T^c$, as in \textit{loc.\ cit.}

We also emphasize the recent work \cite{BHSZ24}, which is closely related to ours and studies analogous questions for big tt-categories via the homological spectrum, obtaining strong results in that setting. By contrast, our arguments are more direct.

Our main result in the direction of the previous question is the following theorem which appears as Corollary~\ref{coro 2}.

\begin{theorem}\label{thm 1.1}
    Let $\mathcal{I}$ be a set.  Let $\{f_i\colon \T\to \T_i\}_{i\in\mathcal{I}}$ be a family of coproduct-preserving tt-functors between tt-categories with small coproducts. Assume that each $f_i$ has a coproduct-preserving adjoint $g_i$ satisfying the projection formula  (see Section \ref{sec: fiberwise building}). If  $\mathbb{1}_{\T}\in \mathrm{Loc}^\otimes(\coprod_{i\in \mathcal{I}}g_{i}(\mathbb{1}_{\T_i}))$, then for any $x$ and $y$ in $\T$, we obtain the following properties.  
    \begin{enumerate}
        \item \textit{(Detection)} If $f_i(x)=0$ for each $i\in \mathcal{I}$, then $x=0$.
        \item \textit{(Building)} $x\in \mathrm{Loc}^\otimes (y)$ if and only if $f_i(x)\in \mathrm{Loc}^\otimes(f_i(y))$ for all $i\in \mathcal{I}$.   
    \end{enumerate}
\end{theorem}

A few remarks are in order. First, we do not specify whether the adjoint $g_i$ is left or right adjoint to $f_i$. The reason is that this distinction is irrelevant for our results, and both situations occur naturally in practice. For example, in the setting of rigidly-compactly generated tt-categories, this is automatic: a geometric functor admits a right adjoint, and the adjunction satisfies the projection formula. By contrast, in the non-rigidly-compactly generated setting, the functor $f_i$ under consideration may instead admit a left adjoint satisfying the corresponding projection formula, as we show in Section~\ref{sec: beyond big}.

On the other hand, note that \textit{detection} is a consequence of \textit{building} in the previous result. Nevertheless, it is worth emphasizing this special case, since it makes explicit that the family of functors $\{f_i\}_{i \in \mathcal{I}}$ is jointly conservative.

As an application of the previous result, we provide a criterion for stratification in the context of rigidly-compactly generated tt-categories and it appears as Corollary \ref{coro:fiberwise stratification} later  in this text.

\begin{theorem}
  Let $\{f_i\colon \T\to \T_i\}_{i\in\mathcal{I}}$ be a family of geometric functors between rigidly-compactly generated tt-categories with weakly Noetherian Balmer spectrum. Assume that  $\mathbb{1}_{\T}\in \mathrm{Loc}^\otimes(\coprod_{i\in \mathcal{I}}g_{i}(\mathbb{1}_{\T_i}))$.
    If  each $\T_i$ is stratified, and  $\coprod_{i\in \mathcal{I}}\mathrm{Spc}(f_i)$ is an injection on Balmer spectra,  then $\T$ is stratified as well.
\end{theorem}

 Let us note that a single-functor version of the previous theorem already appears in \cite[Theorem 17.20]{BCHS23b} under the name of nil-descent (see also \cite{Bar21}). However, in the presence of infinitely many non-trivial categories $\T_i$, one cannot in general reduce to a single functor, since an infinite product of big tt-categories is rarely itself big (see, for example, \cite[Remark 2.7]{Gom24a}).

Furthermore, the significance of the previous theorem lies partly in the fact that stratification, in the sense of \cite{BHS21}, is less restrictive than BIK-stratification, which requires a central action of a suitable graded ring. In practice, such a graded ring typically arises as the graded endomorphism ring of the monoidal unit of the given big tt-category. However, there are examples where this ring is not sufficient for the intended purpose; see, for instance, \cite{BG25}.

As an application, we establish stratification for the big derived category of permutation modules over a finite group and for the category of representations of a finite group scheme, both over an arbitrary Noetherian base. The former result is new in this level of generality and relies on the field case proved by Balmer and Gallauer in \cite{BG25}, whereas the latter was recently obtained in \cite{BBIKP} by different methods.

In both cases, our strategy is similar: we combine Neeman’s stratification theorem for the derived category of a commutative Noetherian ring with the corresponding stratification results over fields. The main results in this direction, stated later as Theorem~\ref{thm: stratification for T} and Theorem~\ref{coro: stratification for rep}, are summarized below.

\begin{theorem}
     Let   $R$ be a   commutative Noetherian  ring, $G$ be a finite group and $\mathbb{G}$ be a finite group scheme defined over $R$. Let $\T$ denote either  the big derived category of permutation $RG$--modules $\T(G,R)$ or  the category of representations $\mathrm{Rep}(\mathbb{G},R)$ defined in \cite{BBIKP}. Then $\T$ is stratified, that is, there is bijection 
   \[
   \{\textrm{Localizing ideals of }\T\}\leftrightarrow \{ \textrm{Subsets of } \mathrm{Spc}(\T^c)\}
   \]
   via the Balmer-Favi support. In particular, $\T$ satisfies the Telescope Property, that is, any smashing localization of $\T$ is generated by its compact part.
\end{theorem}

It should be noted that the previous result requires some information about the topology of the Balmer spectrum of $\T^c$. When $\T = \mathrm{Rep}(\mathbb{G}, R)$, the essential input is the cohomological finite generation property for finite group schemes as shown in \cite{vdk23}. For $\T = \T(G, R)$, we demonstrate in Section \ref{sec:Noetherianity} that its Balmer spectrum is indeed Noetherian and show that the base change functors to residue fields induce a set partition of the spectrum.

\subsection*{Acknowledgments} I am deeply grateful to Paul Balmer for an inspiring discussion. I also wish to thank Martin Gallauer and Manuel Hoff for answering several questions, Henning Krause for his encouragement and interest in this project, and Tobias Barthel for pointing out several relevant references.
 This work was supported by the Deutsche Forschungsgemeinschaft (Project-ID 491392403 – TRR 358).
%%%%%%%%%%%%%%%%%%%%%%%%%%%%%------------- preliminaries 

\section{Preliminaries}\label{sec:preliminaries}

In this section, we recall basic terminology from tensor-triangular geometry that will appear throughout this work, with the sole goal of fixing notation and establishing relevant conventions. Our primary references are \cite{Bal05}, \cite{BF11}, \cite{stevenson2018tour}, and \cite{BHS21}. In particular, we adopt the convention of using `tt' as an abbreviation for `tensor-triangulated' or `tensor-triangular', depending on the context.

\begin{recollection} 
A \textit{big tt-category $\T$} is a rigidly compactly generated tensor-triangulated category $\T$ (e.g. see \cite{BF11} or \cite[Section 2]{stevenson2018tour}). We write $\T^c$ to denote the rigid-compact part of $\T$. A \textit{geometric functor} $f^\ast\colon \T\to \T'$ is a strongly monoidal triangulated functor between big tt-categories that commutes with coproducts. In particular, we get for \textit{free} that $f^\ast$ has a right adjoint $f_\ast$, which itself has a right adjoint $f^!$. Moreover, the adjunction $f^\ast\dashv f_\ast$ satisfies the projection formula; that is, for $x$ in $\T$ and $y$ in $\T'$, we have
\[
f_\ast(f^\ast(x)\otimes y)\cong x\otimes f_\ast(y).
\]
We refer to \cite[Theorem 1.3]{BDS} for further details. 
\end{recollection}

\begin{notation}\label{big cat geom functors}
    Let $\T$ be a tt-category with small coproducts. For a collection of objects $\mathcal{X}$ in $\T$, we write $\mathrm{Loc}(\mathcal{X})$ and $\mathrm{Loc}^\otimes(\mathcal{X})$ to denote the localizing subcategory and the localizing tensor-ideal generated by $\mathcal{X}$, respectively. 
\end{notation}

\begin{remark}
    Let $f\colon \T\to \T'$ be a coproduct-preserving triangulated functor between tt-categories with small coproducts. Then
$f(\mathrm{Loc}(\mathcal{X}))\subseteq \mathrm{Loc}(f(\mathcal{X}))$.
Moreover, if $f$ is additionally strongly monoidal, then
$f(\mathrm{Loc}^\otimes(\mathcal{X}))\subseteq \mathrm{Loc}^\otimes(f(\mathcal{X}))$.
\end{remark}

For an essentially small tt-category $\K$, we denote its Balmer spectrum by $\mathrm{Spc}(\K)$. Now let $\T$ be a big tt-category. When $\mathrm{Spc}(\T^c)$ is weakly Noetherian, an abstract theory of stratification has been developed in \cite{BHS21}, with \cite{Ste13} and \cite{stevenson2017local} as precursors. Roughly speaking, tt-stratification provides a classification of localizing tensor-ideals in $\T$ by subsets of $\mathrm{Spc}(\T^c)$. We briefly recall the relevant definitions used in this document.

It is important to emphasize, however, that Benson-Iyengar-Krause \cite{BIK11} introduced a different notion of stratification. This approach is more powerful because it simultaneously classifies localizing tensor-ideals in $\T$ and thick tensor-ideals in $\T^c$ in terms of certain subsets of the homogeneous spectrum of a specific graded ring. However, this method requires stronger assumptions on $\T$, and in our setting it is unclear whether these conditions hold. For example, in the case of the derived category of permutation modules over $G$, there is no evident ring that serves as a candidate for BIK-stratification, at least to our knowledge.

From now on, we use `stratification' to refer to the concept in the sense of Barthel-Heard-Sanders, unless specified otherwise.

\begin{recollection} 
Let $\T$ be a big tt-category. The space $\mathrm{Spc}(\T^c)$ is called \textit{weakly Noetherian} if, for any prime $\P$ in $\T^c$, there exist Thomason subsets $Y_1,\, Y_2\subseteq \mathrm{Spc}(\T^c)$, such that $\{\P\}=Y_1\cap Y_2^c$. As expected, if $\mathrm{Spc}(\T^c)$ is Noetherian, then it is weakly Noetherian. This is nothing but asking that the Hochster dual of $\mathrm{Spc}(\T^c)^\vee$ is $T_D$, that is, all its points are locally closed.  For further details, see \cite{BHS21}. 
\end{recollection} 

For the rest of this section, fix a big tt-category $\T$ such that its Balmer spectrum is weakly Noetherian. 

\begin{recollection}
        Let $\mathcal{P}$ be a point in $\mathrm{Spc}(\T^c)$, and consider two Thomason subsets $Y_1,\, Y_2\subseteq \mathrm{Spc}(\T^c)$ such that $\{\mathcal{P}\}=Y_1\cap Y_2^c$. We write $\mathbf{g}(\mathcal{P})$ (or $\mathbf{g}_\T(\mathcal{P})$ if we need to emphasize the role of $\T$) to denote $e(Y_1)\otimes_R f(Y_2)$, where $e(Y_1)$ and $f(Y_2)$ are the idempotents associated to the Thomason subsets $Y_1$ and $Y_2$, respectively. For instance, when $\mathrm{Spc}(\T^c)$ is Noetherian, both subsets $\overline{\{\mathcal{P}\}}$ and 
        $Y_\mathcal{P}=\{\mathcal{Q}\in \mathrm{Spc}(\K(G;R))\mid \mathcal{P}\not\subseteq\mathcal{Q}\}$
        are Thomason and satisfy the property that their intersection is the singleton containing $\mathcal{P}$. In any case, the object $\mathbf{g}(\mathcal{P})$ is independent of the choice of $Y_1$ and $Y_2$. The \textit{Balmer-Favi support} of an object $x$ in $\T$ is defined by
        \[
        \mathrm{Supp}(x)=\{\mathcal{P}\in \mathrm{Spc}(\T^c)\mid \mathbf{g}(\mathcal{P})\otimes x\not=0\}.
        \]
        See \cite{BF11} and \cite{BHS21} for further details.
\end{recollection}

\begin{recollection}
    The category $\T$ is \textit{stratified} if it satisfies the following properties. 
    \begin{enumerate}
        \item \textit{The local-to-global principle:} For any $x$ in $\T$, it holds that
        \[
        x\in\mathrm{Loc}^\otimes( \mathbf{g}(\mathcal{P})\otimes x\mid \mathcal{P}\in \mathrm{Spc}(\T^c) ).
        \]
        \item \textit{Minimality:} For all primes $\mathcal{P}$ in $\T^c$, the localizing ideal $\mathrm{Loc}^\otimes( \mathbf{g}(\mathcal{P}))$ is generated by any nonzero object it contains.
    \end{enumerate}
    By \cite[Theorem A]{BHS21}, this is equivalent to obtaining a bijection
    \[
    \{\textrm{localizing ideals of }\T\}\leftrightarrow \{\textrm{subsets of } \mathrm{Spc}(\T^c)\}
    \]
    via the map $\mathcal{L}\mapsto \bigcup_{x\in \mathcal{L}}\mathrm{Supp}(x)$.
\end{recollection}

\section{Fiberwise building}\label{sec: fiberwise building}
Throughout, we let $\mathcal{I}$ denote a set. Some of the results of this section hold in the generality of (closed) tt-categories with small coproducts in which the tensor product commutes with coproducts in each variable separately, and are therefore of independent interest. Hence, we choose to work in this level of generality and restrict to big tt-categories only when necessary.  

\begin{hypothesis}\label{hyp: tt-cats with small coproducts}
Let $f\colon \D \to \D'$ be a tt-functor between tt-categories with small coproducts. We assume that $f$ preserves coproducts and has a coproduct-preserving adjoint $g\colon \D' \to \D$ satisfying the projection formula: for $x$ in $\D$ and $y$ in $\D'$, we have
\[
g(f(x)\otimes y)\cong x\otimes g(y).
\]
\end{hypothesis}

\begin{remark}
As mentioned earlier, the previous hypothesis is automatic for geometric functors (see Recollection \ref{big cat geom functors}). On the other hand, note that in Hypothesis \ref{hyp: tt-cats with small coproducts}, we do not specify whether $g$ is a left or right adjoint of $f$. This is because, while for geometric functors $g$ is typically the right adjoint of $f$, this is not necessarily the case beyond big tt-categories, as we will see in Section \ref{sec: beyond big}; in those cases, $g$ is the left adjoint of $f$.
\end{remark}

\begin{proposition}\label{coro1}
Let $\{f_i\colon \D\to \D_i\}_{i\in\mathcal{I}}$ be a family of tt-functors between tt-categories with small coproducts satisfying Hypothesis \ref{hyp: tt-cats with small coproducts}, and let $g_i$ denote the corresponding adjoint of $f_i$. If
\[
\mathbb{1}_{\D}\in \mathrm{Loc}^\otimes\left(\coprod_{i\in \mathcal{I}}g_{i}(\mathbb{1}_{\D_i})\right),
\]
then for any $x$ in $\D$ we have 
\[
\mathrm{Loc}^\otimes(x)=\mathrm{Loc}^\otimes\left(\coprod_{i\in \mathcal{I}}g_{i}(f_{i}(x))\right).
\]

\end{proposition}

\begin{proof}
  First, recall that for a tensor-ideal $\mathcal{X}$ in $\D$, we have that $\mathrm{Loc}^\otimes (\mathcal{X})=\mathrm{Loc}(\mathcal{X})$. In particular, 
  \[
  \mathrm{Loc}^\otimes \left(\coprod_{i\in \mathcal{I}}g_{i}(\mathbb{1}_{\D_i})\right)=\mathrm{Loc}\left(\left\langle \coprod_{i\in \mathcal{I}}g_{i}(\mathbb{1}_{\D_i})\right\rangle^\otimes\right)
  \]
  where $\langle\mathcal{Y}\rangle^\otimes$ denotes the tensor-ideal generated by the set $\mathcal{Y}$. Since $x\otimes -$ is a triangulated functor that commutes with coproducts, we obtain that 
  \[
  x\otimes \mathbb{1}\in  \mathrm{Loc} \left( x \otimes \left\langle \coprod_{i\in \mathcal{I}} g_{i}(\mathbb{1}_{\D_i}) \right\rangle^\otimes \right).  
  \]
   But note that $x\otimes \langle \coprod_{i\in \mathcal{I}} g_{i}(\mathbb{1}_{\D_i}) \rangle^\otimes =\langle x\otimes \coprod_{i\in \mathcal{I}} g_{i}(\mathbb{1}_{\D_i}) \rangle^\otimes  $. Therefore,  $$ x\in  \mathrm{Loc}\left(\left\langle x\otimes \coprod_{i\in \mathcal{I}} g_i(\mathbb{1}_{\D_i})\right\rangle ^\otimes \right)= \mathrm{Loc}^\otimes \left( x\otimes \coprod_{i\in \mathcal{I}} g_{i}(\mathbb{1}_{\D_i}) \right) \subseteq \mathrm{Loc}^\otimes(x). $$
  It follows that $\mathrm{Loc}^\otimes(x\otimes \coprod_{i\in \mathcal{I}} g_{i}(\mathbb{1}_{\D_i}))= \mathrm{Loc}^\otimes(x)$. 
  Moreover, by the projection formula we obtain $x\otimes g_{i}(\mathbb{1}_{\D_i})\cong g_{i}(f_{i}(x))$. Thus $$x\otimes \coprod_{i\in \mathcal{I}} g_{i}(\mathbb{1}_{\D_i})\cong \coprod_{i\in\mathcal{I}}g_{i}(f_{i}(x)).$$ Hence the result follows. 
\end{proof}

\begin{corollary}\label{coro 2}
Let $\{f_i\colon \D\to \D_i\}_{i\in\mathcal{I}}$ be a family of tt-functors between tt-categories with small coproducts satisfying Hypothesis \ref{hyp: tt-cats with small coproducts}, and let  $g_i$ denote the corresponding adjoint of $f_i$. If  $\mathbb{1}_{\D}\in \mathrm{Loc}^\otimes(\coprod_{i\in \mathcal{I}}g_{i}(\mathbb{1}_{\D_i}))$, then for any $x$ and $y$ in $\D$, we obtain the following properties.  
    \begin{enumerate}
        \item \textit{(Detection)} If $f_i(x)=0$ for each $i\in \mathcal{I}$, then $x=0$.
        \item \textit{(Building)} $x\in \mathrm{Loc}^\otimes (y)$ if and only if $f_i(x)\in \mathrm{Loc}^\otimes(f_i(y))$ for all $i\in \mathcal{I}$.   
    \end{enumerate}
\end{corollary}

\begin{proof}
Note that Detection is a particular case of Building. Therefore it is enough to prove Building. The left to right implication is clear since each $f_i$ is assumed to be a tt-functor. We will prove the converse. Since $ f_i(x)\in \mathrm{Loc}^\otimes (f_i(y))=\mathrm{Loc}(f_i(y)\otimes \D_i)$ and $g_i$ commutes with coproducts, we obtain that 
\begin{align*}
g_i(f_i(x))\in \mathrm{Loc}(g_i(f_i(y)\otimes \D_i))= & \mathrm{Loc}(y\otimes g_i(\D_i))\\  \subseteq & \mathrm{Loc}^\otimes(y).     
\end{align*}
 Here we used the projection formula  for the first equality. We deduce that 
 \[
 \mathrm{Loc}^\otimes(g_i(f_i(x)))\subseteq \mathrm{Loc}^\otimes (y), \quad\mbox{ for all } i\in \mathcal{I}.
 \]
  By Proposition~\ref{coro1}, we get
\begin{align*}
    x\in  \mathrm{Loc}^\otimes \left(\coprod_{i\in \mathcal{I}}(g_i(f_i(x)))\right) 
      \subseteq   \mathrm{Loc}^\otimes(y).
\end{align*}
This completes the proof. 
\end{proof}

\begin{remark}
    In fact, we can rephrase the previous result as follows. For a tt-category $\D$ with small coproducts, we let $\mathbf{Loc}^{\otimes}(\D)$ denote the collection of all localizing $\otimes$-ideals of $\D$. In the context of Corollary~\ref{coro 2}, we obtain that 
    \[
   \prod_{i\in \mathcal{I}}\mathbf{Loc}^\otimes(f_i)\colon \mathbf{Loc}^{\otimes}(\D) \to  \prod_{i\in \mathcal{I}}\mathbf{Loc}^{\otimes}(\D_i)
    \]
    is injective.
\end{remark}

\begin{remark}
    If one is interested in colocalizing hom-ideals instead, one could consider the dual of Hypothesis \ref{hyp: tt-cats with small coproducts}, assuming that the adjunctions satisfy \cite[Equation 2.19]{BDS} (e.g., geometric functors), and adapt the proof of Corollary~\ref{coro 2} to obtain an analogous result. We leave the details to the interested reader. 
\end{remark}

\subsection{Consequences for stratification} 

From this point onward, we restrict our discussion to big tt-categories and geometric functors. In order to keep the notation manageable, for a geometric functor $f_i$, we write $g_i$ to denote its right adjoint.

\begin{lemma}\label{descent for local to global}
    Let $\{f_i\colon \D\to \D_i\}_{i\in\mathcal{I}}$ be a family of geometric functors between big tt-categories with weakly Noetherian Balmer spectrum. Assume that 
    \[
    \mathbb{1}_{\D}\in \mathrm{Loc}^\otimes\left(\coprod_{i\in \mathcal{I}}g_{i}(\mathbb{1}_{\D_i})\right).
    \]
    If each $\D_i$ satisfies the local to global principle, then so does $\D$.
\end{lemma}
\begin{proof}
  Set $\varphi_i:=\mathrm{Spc}(f_i)$.  Let $\P$ be a triangular prime in $\D^c$. Recall that $f_i(\mathbf{g}\P)=\mathbf{g}\varphi^{-1}_i(\P)$ for the weakly visible set $\varphi_i^{-1}(\P)$. Now, fix $i\in \mathcal{I}$. We have  
   \begin{align*}
       \mathrm{Loc}^\otimes(f_i(\mathbf{g}\P)\mid \P\in \mathrm{Spc}(\D^c)) & =  \mathrm{Loc}^\otimes(\mathbf{g}\varphi_i^{-1}\P\mid \P\in \mathrm{Spc}(\D^c)) \\ & = \mathrm{Loc}^\otimes(\mathbf{g}\mathcal{Q}\mid \mathcal{Q}\in \mathrm{Spc}(\D_i^c))
   \end{align*}
    where the last equality holds by the local to global principle for $\D_i$ and the fact that any $\mathbf{g}\mathcal{Q}$, for $\mathcal{Q}$ in $\mathrm{Spc}(\D_i^c)$, must contained in the localizing ideal generated by  some $\mathbf{g}\varphi_i^{-1}\mathcal{P}$. It follows that $f_i(\mathbb{1})$ must lie on the left hand side. By building (see Corollary~\ref{coro 2}), we deduce that $\mathbb{1}\in \mathrm{Loc}^\otimes(\mathbf{g}\P\mid \P\in \mathrm{Spc}(\D^c))$. This completes the proof. 
    \end{proof}

\begin{theorem}\label{thm: fiberwise}
 Let $\{f_i\colon \D\to \D_i\}_{i\in\mathcal{I}}$ be a family of geometric functors between big tt-categories. Assume that the following properties hold:
   \begin{enumerate}
       \item $\coprod_{i\in \mathcal{I}}\mathrm{Spc}(f_i)\colon \coprod_{i\in \mathcal{I}}\mathrm{Spc}(\D_i^c)\to \mathrm{Spc}(\D^c)$ is injective. 
       \item $\D_i$ satisfies minimality for each $i\in \mathcal{I}$.
       \item  $\mathbb{1}_{\D}\in \mathrm{Loc}^\otimes(\coprod_{i\in \mathcal{I}}g_i(\mathbb{1}_{\D(i)}))$.
   \end{enumerate}
     Then $\D$ satisfies minimality.
\end{theorem}

 \begin{proof}
        We will  show minimality of $\mathrm{Loc}^\otimes (\mathbf{g}_\D(\mathcal{P}))$ for each prime $\mathcal{P}$ in $\D^c$. First, fix $i$ in $\mathcal{I}$. 
        Since $f_i$ induces an injective morphism  
        \[
        \varphi_i:=\mathrm{Spc}(f_i)\colon\mathrm{Spc}(\D_i^c)\to \mathrm{Spc}(\D^c)
        \]
        we can invoke \cite[Proposition 3.12]{BHS21}, that is, we obtain that  
      \[
      f_i(\mathbf{g}_{\D}(\mathcal{P}))\simeq \left\{
        \begin{array}{ll}
           
            0 & \textrm{ if } \mathcal{P}\not\in \mathrm{im}(\varphi_i) \\
            \mathbf{g}_{\D_i}(\mathcal{Q}) & \textrm{ if } \varphi_i(\mathcal{Q})=\mathcal{P}.
        \end{array}
    \right.
      \]
Let $x\in \mathrm{Loc}^\otimes (\mathbf{g}_{\D}(\mathcal{P}))$.  Note that $\mathrm{Loc}^\otimes (f_i(x))=\mathrm{Loc}^\otimes (f_i(\mathbf{g}_\D(\mathcal{P})))$. Indeed, if $\mathcal{P}\not\in \mathrm{im}(\varphi_i)$, then both localizing subcategories are trivial; if $\varphi_i(\mathcal{Q})=\mathcal{P}$, then the equality follows by minimality for $\D_i$. Our assumption $(3)$ allows to invoke Corollary~\ref{coro 2}. In particular, by building  we deduce that $\mathrm{Loc}^\otimes (x)=\mathrm{Loc}^\otimes (\mathbf{g}_\D(\mathcal{P}))$ as we wanted.
 \end{proof}

 Let us record the following easy but relevant consequence of Theorem \ref{thm: fiberwise} combined with Lemma \ref{descent for local to global}.

 \begin{corollary}\label{coro:fiberwise stratification}
     Let $\{f_i\colon \D\to \D_i\}_{i\in \mathcal{I}}$ be a family of geometric functors between big tt-categories with weakly Noetherian Balmer spectrum. Assume that  the family of functors satisfy the hypotheses of Theorem \ref{thm: fiberwise}. If we additionally assume that each $\D_i$ is stratified, then $\D$ is stratified as well. 
 \end{corollary}

\subsection{A criterion for detecting building}

In this section, we establish a method that will allow us to detect building from categories that already satisfy this property. We first need some preparation. 

\begin{hypotheses}\label{Hypotehes}
 Let $\mathcal{I}$ be a set. For each $i\in \mathcal{I}$, assume that we have a commutative square of big tt-categories and geometric functors 
 \begin{center}
    \begin{tikzcd}
        \D   \arrow[r,"f_i"]  \arrow[d,"\tau"'] &     \D_i \arrow[d,"\tau_i"]   \\ 
       \T \arrow[r,"g_i"'] & \T_i 
    \end{tikzcd}
\end{center}
Let $u_i$ and $v_i$ denote the right adjoints of $f_i$ and $g_i$, respectively. Additionally, we assume that: 
\begin{enumerate}
    \item The above square is \textit{horizontally right adjointable at the monoidal unit}; that is,
    \[
     v_i (\tau_i (\mathbb{1}_{\D_i})) \cong \tau (u_i (\mathbb{1}_{\D_i})).
    \]
    \item The family of functors $\{f_i\colon \D\to \D_i\}_{i \in \mathcal{I}}$ builds the tensor unit in $\D$; explicitly,
    \[
    \mathbb{1}_{\D}\in \mathrm{Loc}^\otimes\left(\coprod_{i\in \mathcal{I}}u_i(\mathbb{1}_{\D_i})\right).
    \]
\end{enumerate}
\end{hypotheses}

\begin{lemma}\label{descending building}
    Under Hypotheses \ref{Hypotehes}, we have
    $$
    \mathbb{1}_{\T}\in \mathrm{Loc}^\otimes\left(\coprod_{i\in \mathcal{I}}v_i(\mathbb{1}_{\T_i})\right).
    $$
\end{lemma}

\begin{proof}
   Since $\tau$ is a tt-functor, condition $(2)$ in Hypotheses \ref{Hypotehes} implies that 
\begin{align*}
     \mathbb{1}_{\T}=\tau(\mathbb{1}_{\D})
     &\in \mathrm{Loc}^\otimes \left(\tau\left(\coprod_{i\in \mathcal{I}}u_i(\mathbb{1}_{\D_i})\right)\right) \\
     &= \mathrm{Loc}^\otimes \left(\coprod_{i\in \mathcal{I}}\tau\bigl(u_i(\mathbb{1}_{\D_i})\bigr)\right) \\
     &= \mathrm{Loc}^\otimes \left(\coprod_{i\in \mathcal{I}} v_i(\mathbb{1}_{\T_i})\right),
\end{align*}
where the last equality follows from Hypotheses \ref{Hypotehes}(1) together with the assumption that each $\tau_i$ is a tt-functor.
\end{proof}

\section{New examples, the derived category of permutation modules}

In this section, we apply the methods developed previously to show that the big derived category of permutation modules over an arbitrary Noetherian base is stratified. It is worth emphasizing that stratification in the field case was established by Balmer and Gallauer (see \cite{BG25}). The key ingredients in our argument are Neeman’s stratification theorem for the derived category of a Noetherian ring (see \cite{Nee}), together with the field-case result.

Our approach is elementary; however, the main subtlety lies in proving that the Balmer spectrum is a weakly Noetherian space. In fact, we will show that it is Noetherian. For the moment, we take this as given and defer the proof to Section \ref{sec:Noetherianity}.

For the remainder of this section, let $R$ be a commutative Noetherian ring and $G$ a finite group.

\begin{recollection}
 A \textit{permutation $RG$--module} is an $RG$-module which has a $G$--equivariant $R$--basis. In other words, permutation  $RG$--modules are $R$--linearization of $G$--sets. We write $\Perm(G;R)$ to denote the full subcategory of $\Mod(RG)$ on permutation $RG$--modules, and let $\perm(G;R)$  denote the full subcategory on finitely generated permutation $RG$--modules.  The category $\Perm(G;R)$ is a tensor category with the tensor product given by $\otimes_R$ endowed with the diagonal action of $G$.
\end{recollection}

\begin{recollection}\label{def:big derived category}
    The \textit{small derived category of permutation $RG$--modules} $\K(G;R)$ is defined as  the idempotent completion of bounded homotopy category of finitely generated permutation $RG$--modules. In symbols:
    \[
    \K(G;R):=\mathbf{K}^b(\perm(G;R))^\natural.
    \] 
    The  \textit{big derived category of permutation $RG$--modules} $\T(G;R)$ is defined as the  localizing subcategory of $\mathbf{K}(\Perm(G;R))$ generated by $\K(G;R)$.  The tensor product on $\Perm(G;R)$ commutes with coproducts and allows us to endow $\T(G;R)$ with a tensor structure. In fact, 
the big derived category of permutation $RG$--modules $\T(G;R)$ is a rigidly-compactly generated tensor triangulated category, and $\K(G;R)$ sits inside $\T(G;R)$ as the rigid-compact part. We refer to \cite[Section 3]{BG21} for further details.  
\end{recollection}

\begin{remark}\label{alternative definition for T}
    There is an alternative definition of the big derived category of permutation $RG$--modules. Namely, it is defined as the localization 
    \[
    \T(G;R):=\mathbf{K}(\Perm(G;R))[\{G\textrm{--quasi-iso}\}^{-1}]
    \]
    where a morphism $f$ is a $G$--quasi-isomorphism if the induced morphism on $H$-fixed points $f^H$ is a quasi-isomorphism for all $H\leq G$ (see \cite[Definition 3.6]{BG21}).   
\end{remark}

\begin{remark}\label{Db is a localization of KG}
 Let us stress that the innocent-looking derived category of permutation $RG$--modules is in fact quite big; it has as a localization Krause's homotopy category of injective $RG$--modules, at least when $R$ is regular.  More precisely,  Balmer and Gallauer \cite[Section 4]{BG22b} showed that there is a localization functor 
 \[
 \Upsilon^G\colon\T(G;R)\twoheadrightarrow \mathbf{K}^{\textrm{perm}}(\mathrm{Inj}(RG))
 \]
 where $\mathbf{K}^{\textrm{perm}}(\mathrm{Inj}(RG))$ is the localizing subcategory of $\mathbf{K}(\mathrm{Inj}(RG))$ generated by $Q_\rho(R(G/H))$, $H\leq G$. Here $Q_\rho$ denotes the right adjoint of the localization functor $Q\colon \mathbf{K}(\mathrm{Inj}(RG))\to \mathbf{D}(RG)$ which inverts quasi-isomorphisms. In particular, $\mathbf{K}^{\textrm{perm}}(\mathrm{Inj}(RG))$ agrees with $\mathbf{K}(\mathrm{Inj}(RG))$ when $R$ is regular.

Moreover, the previous localization  restricts to a localization on compact objects
    \[
    \K(G;R)\twoheadrightarrow \mathbf{D}^{\textrm{perm}}(RG)
    \]
      where $\mathbf{D}^{\textrm{perm}}(RG)$ denotes the thick subcategory of $\mathbf{D}^b(RG)$ generated by $R(G/H)$, $H\leq G$. This has an immediate consequence on Balmer spectra, namely this localization induces an injective map 
      \[
      \mathrm{Spc}(\mathbf{D}^{\textrm{perm}}(RG))\hookrightarrow \mathrm{Spc}(\K(G;R)).
      \]
      In particular,  this map is open when $R$ is a field in modular characteristic; see \cite[Proposition 3.22]{BG25} and the references therein.  
\end{remark}

%%%%%%%%%%%%%%%%%%%%%5--------------

Our strategy is to construct families of functors relating $\T(G;R)$ and the unbounded derived category of $R$, $\mathbf{D}(R)$, that satisfy Hypotheses \ref{Hypotehes}. As noted above, we will make use of Neeman’s stratification result for $\mathbf{D}(R)$; for the reader’s convenience, we recall the statement here. See \cite{Nee} for further details.

\begin{theorem}[Neeman]\label{thm of neeman}
Let $R$ be a commutative Noetherian ring. For each $\pp\in \mathrm{Spec}(R)$, consider the geometric tt-functor $f_{\pp}\colon \mathbf{D}(R)\to \mathbf{D}(k(\pp))$ obtained by base change along the residue field $R\to k(\pp)$ at $\pp$. Let $u_{\pp}$ denote the right adjoint of $f_{\pp}$. Then 
\[
R=\mathbb{1}_{\mathbf{D}(R)}\in \mathrm{Loc} \left( \coprod_{\pp\in \mathrm{Spec}(R)} u_{\pp}( \mathbb{1}_{\mathbf{D}(k(\pp))} ) \right).
\]    
\end{theorem}

\begin{recollection}\label{rec:comm diagram for T}
  Consider the functor 
  \[
  \mathrm{Infl}_1^G\colon\mathbf{D}^\mathrm{perf}(R) \to \K(G;R)
  \]
  induced by the augmentation map $RG\to R$. We can also think about this map as the inflation functor induced by the morphism of groups $G\to G/G$ which explains the notation. One verifies that this functor is fully faithful and tensor triangulated. Hence it induces a fully faithful tt-functor 
  \[
  \mathrm{Infl}_1^G\colon\mathbf{D}(R)\to \T(G;R).
  \]
    On the other hand, for $\pp\in \mathrm{Spec}(R)$, it is not too hard to verify that the functor $k(\pp)\otimes_R-$ induced by base along the residue field $R\to k(\pp)$ takes permutation $RG$--modules into permutation $k(\pp)G$--modules. In particular, we obtain a geometric functor $$g_\pp\colon \T(G;R)\to \T(G;k(\pp)).$$ We will write $g_{\pp,G}$ if we need to emphasize the role of the group $G$. Consider the following diagram where each pair of arrows is an adjunction following the convention that the top arrow from the pair is the left adjoint. 
    \begin{equation}\label{commutative diagram for T}
    \begin{tikzcd}
        \mathbf{D}(R)   \arrow[r,"u_\pp" below, <-,shift left=0] \arrow[r,"f_\pp",->,shift left=2]   \arrow[d, "\mathrm{Infl}" left, ->,shift left=0] \arrow[d, "\mathrm{\varphi_R^G}", <-,shift left=2]  &     \mathbf{D}(k(\pp)) \arrow[d, "\mathrm{Infl}" left, ->,shift left=0] \arrow[d, "\mathrm{\varphi_\pp^G}", <-,shift left=2]   \\ 
       \T(G;R)    \arrow[r,"v_{\pp,G}" below, <-,shift left=0] \arrow[r,"g_{\pp.G}",->,shift left=2] & \T(G;k(\pp)). 
    \end{tikzcd}
  \end{equation}
 here $f_\pp$ is short for $k(\pp)\otimes_R^L-$ which is indeed a geometric functor. In particular, we identify $g_{\pp,1}$ and $v_{\pp,1}$ with $f_\pp$ and $u_\pp$, respectively.  
 Moreover, we have that   $\mathrm{Infl}_1^G$ is compatible with $g_{\pp,G}$ in the sense that the previous diagram  is commutative from the top left to the bottom right.
\end{recollection}

\begin{proposition}\label{prop:adjointable square for T}
    In the context of Recollection \ref{rec:comm diagram for T}, we have that   
    \[
    \mathrm{Infl} (u_\pp (\mathbb{1}_{\mathbf{D}(k(\pp))}))\simeq v_{\pp,G}(\mathrm{Infl}(\mathbb{1}_{\mathbf{D}(k(\pp))}))
    \]
    for each $\pp\in \mathrm{Spec}(R)$.
\end{proposition}

\begin{proof}
Fix $\pp\in \mathrm{Spec}(R)$. In order to ease notation, we write $k$ to denote $k(\pp)$, and drop $\pp$ from the notation involving the  functors from Recollection \ref{rec:comm diagram for T}. Since the functor $\mathrm{Infl}$ is fully faithful, it is enough to show that 
\[
\eta\colon \mathrm{Infl}(\varphi_R^G(v_G(k)))\to v_G(k)
\]
is an equivalence, where $\eta$ denotes the  
counit of the adjunction $(\mathrm{Infl},\varphi_R^G)$ evaluated at $v_G(k)$. 

In view of  Remark \ref{alternative definition for T}, it is enough to verify that $\eta$ is a $G$--quasi-isomorphisms, that is, the induced map on  $H$-fixed points $\eta^H$ is a quasi-isomorphism for all $H\leq G$. First, note that the right adjoint of the inflation functor $\mathrm{Infl}_1^G\colon \mathbf{D}(R)\to \T(G;R)$ is given by taking ordinary $G$-fixed points degree-wise $(-)^G\colon \T(G;R)\to \T(1;R)$, and then inverting quasi-isomorphisms $\T(1;R)\to \mathbf{D}(R)$. In other words, $\eta^G$ is a quasi-isomorphism if $\varphi_R^G(\eta)$ is an isomorphism. 
Hence, $\eta$ would be a $G$--quasi-isomorphism if its restriction to $\T(H;R)$ is an isomorphism in $\mathbf{D}(R)$ under $\varphi_R^H$ for each subgroup $H$ of $G$.

Note that for any subgroup $H\leq G$, we have an analogous commutative square as in \ref{commutative diagram for T}. In particular, get an isomorphism  
\begin{equation}
    \label{eq-1}
    u(k)\simeq \varphi_R^H( v_H(k)).
\end{equation}
Now, fix a subgroup $H$ of $G$. Consider the following diagram 
  \begin{center}
    \begin{tikzcd}
 \T(G;R)   \arrow[d, "\mathrm{ind}" left, <-,shift left=0] \arrow[d, "\mathrm{res}", ->,shift left=2] \arrow[r,"v_G" below, <-,shift left=0] \arrow[r,"g_G",->,shift left=2] & \T(G;k)   \arrow[d, "\mathrm{ind}" left, <-,shift left=0] \arrow[d, "\mathrm{res}", ->,shift left=2] \\  \T(H;R) \arrow[r,"v_{H}" below, <-,shift left=0] \arrow[r,"g_{H}",->,shift left=2] &  \T(H;k)
    \end{tikzcd}
\end{center} 
 Note that the left adjoints from the left-bottom corner to the right-top corner commute. Indeed, one can check this on compact objects using that $\mathrm{ind}_H^G\mathrm{ind}_L^H(x)\cong \mathrm{ind}_L^G(x)$ for any subgroup $L$ of $H$. It follows that the right adjoint functors in the opposite direction also commute. Thus, we have a natural isomorphism
 \begin{equation}
     \label{eq-2}
     \mathrm{res}^G_H\circ v_G \simeq v_{H}\circ \mathrm{res}^G_H.
 \end{equation}
  Moreover, note that for any subgroup $H\leq G$, we have an identification $\mathrm{Infl}_1^H\simeq \mathrm{res}^G_H\circ \mathrm{Infl}_1^G$. By this observation,  we have that 
  \[  \mathrm{res}^G_H\mathrm{Infl}_1^G(\varphi_R^Gv_G(k))\simeq \mathrm{Infl}_1^H(\varphi_R^Gv_G(k)))\simeq \mathrm{Infl}_1^H(\varphi_R^Hv_H(k))
  \]
  where the equivalence from the right is obtained from Equation \ref{eq-1}; indeed note that $\varphi^H_R(v_H(k))\simeq u(k)\simeq  \varphi_R^G(v_G(k))$. On the other hand, by Equation \ref{eq-2}, we obtain that $\mathrm{res}^G_H(v_G(k))\simeq v_H(k)$. Using these equivalences, we can identify $\mathrm{res}^G_H(\eta)$ with the counit of the adjunction $(\mathrm{Infl}_1^H,\varphi^H_R)$ evaluated at $v_H(k)$. In particular, it follows that $\varphi^H_R(\mathrm{res^G_H(\eta)})$ is an isomorphism as we wanted. 
\end{proof}

\begin{theorem}\label{cor:joitly conservaty for T}
    Let $G$ be a finite group and $R$ be a commutative Noetherian ring. Consider the same notation as in Recollection \ref{rec:comm diagram for T}. Then 
    \[
    R\in \mathrm{Loc}^\otimes\left( \coprod_{\pp\in \mathrm{Spec}(R)} u_\pp(k(\pp)) \right)\subseteq \T(G;R).
    \]
    In particular, the family of functors 
    \[
    \{g_\pp\colon \T(G;R)\to \T(G;k(\pp))\}_{\pp\in\mathrm{Spec}(R)}
    \]
    is jointly conservative. 
\end{theorem}

\begin{proof}
    The first claim follows by Lemma \ref{descending building} applied to   Proposition \ref{prop:adjointable square for T}. The second claim follows by Corollary~\ref{coro 2}. 
\end{proof}

Our main application is the following.  

\begin{theorem}\label{thm: stratification for T}
    Let $R$ be a commutative Noetherian ring and $G$ be a finite group. Then the big derived category of permutation $RG$--modules $\T(G;R)$ is stratified. In particular, it satisfies the Telescope Property. 
\end{theorem} 

\begin{proof}
In Corollary  \ref{cover of spc} we  show that the Balmer spectrum of $\K(G;R)$ is Noetherian.  By \cite{BG25} we have that $\T(G;k(\pp))$ is stratified for every residue field $k(\pp)$ of $R$.  That the family of functors  
\[
\{g_\pp\colon \T(G;R)\to \T(G;k(\pp))\}_{\pp\in\mathrm{Spec}(R)}
\]
satisfies condition $(3)$ of Theorem \ref{thm: fiberwise} is established in Theorem \ref{cor:joitly conservaty for T}. The map $\coprod_{\pp\in\mathrm{Spec}(R)}\mathrm{Spc}(g_\pp)$ is a bijection by Theorem \ref{thm-setpartition-general}. We now have all the ingredients required to invoke Corollary \ref{coro:fiberwise stratification} which gives us stratification. 

 The second claim is a well-known consequence of stratification; see \cite{BHS21}.
 \end{proof}

\section{Old examples, the category of representations of a finite group scheme}\label{sec:group schemes}

In this section, we give an alternative proof of a recent result from \cite{BBIKP} on stratification for the category of representations of a finite group scheme over a commutative Noetherian ring. We emphasize that, in \textit{loc.\ cit.}, the authors establish BIK-stratification (see \cite{BIK11}), which is stronger than the notion of stratification considered in this work. In particular, BIK-stratification yields a description of the Balmer spectrum of the compact objects in the associated big tt-category. Nevertheless, combining the main results of this section with \cite{BCHS23}, we recover a description of the Balmer spectrum without appealing to BIK-stratification.

Our strategy follows that of the previous section. The key ingredients are Neeman’s theorem together with stratification for the category of representations in the field case. For the reader’s convenience, we recall the relevant definitions here.

\begin{recollection}
Let $R$ be a commutative Noetherian ring and $G$ be a finite group scheme over $R$. Recall that the group algebra $RG$  is the dual of the coordinate algebra $R[G]$. We write $\mathrm{Mod}(G;R)$ to denote the full subcategory of $RG$--modules whose objects are lattices, that is, $RG$--modules that are projective as $R$--modules. We will consider $\mathrm{Mod}(G;R)$ as an exact category with the  exact structure determined by the split exact structure of $\mathrm{Proj}(R)$. We will also endow $\mathrm{Mod}(G;R)$ with the symmetric monoidal structure induced by $\otimes_R$ equipped with diagonal $G$--action.
Let $\mathrm{mod}(G;R)$ denote the full subcategory of $\mathrm{Mod}(G;R)$ consisting of finitely generated lattices. The \textit{category of $RG$--representations}, $\mathrm{Rep}(G;R)$, is defined as (the homotopy category of) the ind-completion of the derived $\infty$-category $\mathbf{D}^b(\mathrm{mod}(G;R))$ of finitely generated $RG$--lattices $\mathrm{mod}(G;R)$. In symbols, 
\[
\mathrm{Rep}(G;R):=\mathrm{ind}\, \mathbf{D}^b(\mathrm{mod}(G;R)).
\]
In fact, $\mathrm{Rep}(G;R)$ is a big tt-category with the monoidal structure induced by $\otimes_R$ and the monoidal unit corresponds to $\mathbf{i}_GR$, the injective resolution of $R$ in $\mathrm{Mod}(G;R)$. In particular: 
\begin{itemize}
    \item When $G=1$, this construction recovers recovers the big derived category $\mathbf{D}(R)$ of $R$.
    \item When $R$ is a field, it recovers $\mathbf{K}(\mathrm{Proj}(kG))$.
\end{itemize}
\noindent For further details, we refer to \cite[Section 5]{BBIKP}.
\end{recollection}

\begin{recollection}\label{forgetful functor}
 Recall that the augmentation map $RG\to R$ induces a restriction functor $\mathrm{proj}(R)\to \mathrm{mod}(G;R)$ which is strongly monoidal since the monoidal structure on the latter is given by $\otimes_R$. In particular, we obtain a tensor triangulated functor 
 \[
 \mathrm{Infl}\colon \mathbf{D}^\textrm{perf}(R) \to \mathbf{D}^b(\mathrm{mod}(G;R)).
 \]
 Since ind-completion is a monoidal functor, we obtain a  tt-functor 
 \[
 \mathrm{Infl}\colon\mathbf{D}(R)\to \mathrm{Rep}(G;R)
 \]
 which commutes with coproducts.

 We will also consider  the functor $\mathbf{q}\colon \mathrm{Rep}(G;R)\to \mathbf{D}(RG)$ that inverts quasi-isomorphisms. In fact, this functor admits a right adjoint $\mathbf{i}_G\colon \mathbf{D}(RG)\to \mathrm{Rep}(G;R)$. In particular, $\mathbf{i}_GR$ is  the injective resolution of $R$ in $\mathrm{Mod}(G;R)$ which justify  the notation for this right adjoint $\mathbf{i}_G$.
\end{recollection}

\begin{recollection}\label{recollection:commutative diagram for Rep}
    For each $\pp$ in  $ \mathrm{Spec}(R)$, base change along the residue field $R\to k(\pp)$ induces a geometric tt-functor $$g_\pp\colon \mathrm{Rep}(G;R)\to \mathrm{Rep}(G_{k(\pp)};k(\pp))$$ where $G_{k(\pp)}$ denotes the group scheme obtained from $G$ by base change along $R\to k(\pp)$. Moreover, combining $g_\pp$ with the functors $\mathrm{Infl}$  from Recollection  \ref{forgetful functor}  we obtain a commutative diagram 
    \begin{center}
    \begin{tikzcd}
        \mathbf{D}(R)   \arrow[r,"u_\pp" below, <-,shift left=0] \arrow[r,"f_\pp",->,shift left=2]    \arrow[d,"\mathrm{Infl}"'] &     \mathbf{D}(k(\pp)) \arrow[d,"\mathrm{Infl}"]   \\ 
       \mathrm{Rep}(G;R) \arrow[r,"v_\pp" below, <-,shift left=0] \arrow[r,"g_\pp",->,shift left=2]   & \mathrm{Rep}(G_{k(\pp)};k(\pp)) 
    \end{tikzcd}
\end{center}
where $f_\pp$ denotes $k\otimes_R^L-$.
\end{recollection}

\begin{proposition}\label{prop:adjointable square for rep}
    In the context of Recollection \ref{recollection:commutative diagram for Rep}, we have that 
    \[
    \mathrm{Infl} (u_\pp (\mathbb{1}_{\mathbf{D}(k(\pp))}))= v_\pp(\mathrm{Infl}(\mathbb{1}_{\mathbf{D}(k(\pp))})).
    \]
    for each $\pp\in\mathrm{Spec}(R)$.
\end{proposition}

\begin{proof}
Fix $\pp\in\mathrm{Spec}(R)$. For simplicity, let us write $k$ instead $k(\pp)$. We have the following diagram 
    \begin{equation}\label{eq-3}
        \begin{tikzcd}
 \mathrm{Rep}(G;R)   \arrow[d, "g_\pp" left, ->,shift left=0] \arrow[d, "v_\pp", <-,shift left=2] \arrow[r,"\mathbf{i}" below, <-,shift left=0] \arrow[r,"\mathbf{q}",->,shift left=2] & \mathbf{D}(RG)  \arrow[d,"-\otimes_R^L k" left, ->,shift left=0]  \arrow[d, "\textrm{inc}", <-,shift left=2] \\  \mathrm{Rep}(G_k;k) \arrow[r,"\mathbf{i_\pp}" below, <-,shift left=0] \arrow[r,"\mathbf{q}_\pp",->,shift left=2] & \mathbf{D}(kG_k)  
    \end{tikzcd}
    \end{equation}
Where $\mathbf{q}$ and $\mathbf{q}_\pp$ are the functors obtained by inverting quasi-isomorphisms, respectively. One verifies that the left adjoints from the top-left corner to the right-bottom corner in the previous square commute, that is, $-\otimes_k^L\circ \mathbf{q} = \mathbf{q}_\pp\circ g_\pp$. Hence the right adjoints in the opposite direction also commute. 

On the other hand, we have a diagram 
 \begin{equation}
     \label{eq-4}
      \begin{tikzcd}
        \mathbf{D}(R)   \arrow[r,"u_\pp" below, <-,shift left=0] \arrow[r,"f_\pp",->,shift left=2]    \arrow[d,"\mathrm{Infl}"'] &     \mathbf{D}(k(\pp)) \arrow[d,"\mathrm{Infl}"]   \\ 
       \mathbf{D}(RG) \arrow[r,"v_\pp" below, <-,shift left=0] \arrow[r,"-\otimes_R^Lk",->,shift left=2]   & \mathbf{D}(kG_{k})) 
    \end{tikzcd}
 \end{equation}
Since both $u_\pp$ and $v_\pp$ are induced by  restriction along $R\to k$, one gets that this square is  horizontally right adjointable; that is, $\mathrm{Infl}\circ u_\pp\simeq v_\pp\circ\mathrm{Infl}$.  

Now, consider the diagram 
\begin{center}
    \begin{tikzcd}
        \mathbf{D}(R)  &  \\ 
          \mathrm{Rep}(G;R)    \arrow[u,"(-)_{hG}"]  \arrow[r,"\mathbf{q}"]      &   \mathbf{D}(RG) \arrow[lu,"(-)_{hG}"'] 
    \end{tikzcd}
\end{center}
where $(-)_{hG}$ denotes the homotopy orbits functor, which is left adjoint to the inflation functor. One verifies that the diagram is commutative. Indeed, since all the functors involved are left adjoints, it suffices to check the claim on compact objects in $\mathrm{Rep}(G;R)$. In this case, $\mathbf{q}$ identifies such compact objects with a full subcategory of $\mathbf{D}(RG)$, and the result follows immediately.
It then follows that the right adjoints also commute, that is, $\mathbf{i}_G \circ \mathrm{Infl} \simeq \mathrm{Infl}$. An analogous statement holds upon replacing $R$ by $k$ and $G$ by $G_k$. The result now follows from this observation together with a straightforward diagram chase using the diagrams in Equations \ref{eq-3} and \ref{eq-4}.
\end{proof}

\begin{theorem}
    \label{building for rep}
    With the same notation as in Proposition \ref{prop:adjointable square for rep}, we have  \[
    \mathbb{1}_{\mathrm{Rep}(G;R)}\in \mathrm{Loc}^\otimes\left(\coprod_{\pp\in \mathrm{Spec}(R)} v_\pp(\mathbb{1}_{\mathrm{Rep}(G_{k(\pp)};k(\pp))}) \right).
    \]
    In particular, for $x$ and $y$ in $\mathrm{Rep}(G;R)$, we have that $x\in \mathrm{Loc}^\otimes(y)$ if and only if  
    \[
    g_\pp (x)\in \mathrm{Loc}^\otimes (g_\pp (y))
    \]
    and hence the family of functors $\{g_\pp\}$ is jointly conservative. 
\end{theorem}

\begin{proof}
     By Proposition \ref{prop:adjointable square for rep}, we deduce that the family of functors $\{g_\pp\}$ and $\{f_\pp\}$ satisfy the  Hypotheses \ref{Hypotehes}. Hence, by  Theorem \ref{thm of neeman} we can invoke Lemma \ref{descending building} which proves the first claim. The second claim is now an immediate consequence of Corollary~\ref{coro 2}.
\end{proof}

\begin{corollary}\label{spectrum of rep}
     Let $G$ be a finite group scheme defined over a commutative Noetherian base $R$. Then the comparison map from \cite{Bal10} 
     \[
     \mathrm{Comp}_G\colon \mathrm{Spc}(\mathbf{D}^b(\mathrm{mod}(G;R)))\to \mathrm{Spec}^h(H^\ast(G;R))
     \]
     is an homeomorphism.  In particular, $\mathrm{Spc}(\mathbf{D}^b(\mathrm{mod}(G;R)))$ is Noetherian. 
\end{corollary}

As mentioned above, this result was previously obtained as a consequence of BIK-stratification for $\mathrm{Rep}(G;R)$, as proved in \cite{BBIKP}. Our aim here is to provide an alternative proof of stratification in this setting. To this end, we still require some topological information about $\mathrm{Spc}(\mathbf{D}^b(\mathrm{mod}(G;R)))$. We therefore explain how to combine our results with those of \cite{BCHS23} to recover the statement, thereby avoiding the use of BIK-stratification.

The argument follows the same lines as \cite[Example 1.6]{BCHS23} in the case of finite groups. Moreover, Remark 1.7 of that paper states the result in full generality, although it relies on \cite[Theorem B]{BBIKP}. The latter, however, is already a consequence of Theorem \ref{building for rep}. For the reader’s convenience, we include the complete argument here.

\begin{proof}[Proof of Corollary~\ref{spectrum of rep}]
By \cite{vdk23} we have that $\mathrm{Rep}(G;R)$ is a Noetherian category.
 By Theorem \ref{building for rep}, we get that the family of functors $g_\pp\colon \mathrm{Rep}(G;R)\to \mathrm{Rep}(G_{k(\pp)};k(\pp))$ indexed over $\mathrm{Spec}(R)$ is jointly conservative. On the other hand, each comparison map 
 \[
 \mathrm{Comp}_\pp\colon \mathrm{Spc}(\mathbf{D}^b(\mathrm{mod}(G_{k(\pp)};k(\pp))))\to \mathrm{Spec}^h(H^\ast(G_{k(\pp)};k(\pp)))
 \]
 is a homeomorphism by stratification for $\mathrm{Rep}(G_{k(\pp)};k(\pp))$ \cite{BIKP}. Moreover, the morphism 
 \[
 \coprod_{\pp\in \mathrm{Spec}(R)} \mathrm{Spec}^h(H^\ast(G_{k(\pp)};k(\pp)))  \to  \mathrm{Spec}^h(H^\ast(G;R))
 \]
 is bijective by \cite[Theorem C]{BBIKP}. Hence we have all the ingredients to invoke \cite[Corollary 1.12]{BCHS23}, and the result follows. 
\end{proof}

\begin{remark}
We stress that the approach in \cite{BBIKP} also relies on van der Kallen's result \cite{vdk23} in a non-trivial way. In fact, it serves as a crucial ingredient in their proof. This result is far from trivial, and the brevity of \textit{loc. cit.} may not fully reflect its significance.
\end{remark}

\begin{theorem}\label{coro: stratification for rep}
    Let $G$ be a finite group scheme defined over a commutative Noetherian base $R$. Then $\mathrm{Rep}(G;R)$ is stratified. In particular, the Telescope Property holds for $\mathrm{Rep}(G;R)$. 
\end{theorem}

\begin{proof}
By \cite{BIKP} we know that $\mathrm{Rep}(G_{k(\pp)};k(\pp))$ is stratified for each $\pp\in \mathrm{Spec}(R)$. By Corollary \ref{spectrum of rep} we know that $\mathrm{Spc}(\mathbf{D}^b(\mathrm{mod}(G;R)))$ is Noetherian and each functor $g_\pp$ induce an injection on Balmer spectra. Hence, Theorem \ref{building for rep} provides the remaining ingredient needed to invoke Corollary \ref{coro:fiberwise stratification}, which completes the proof.
\end{proof}

\section{Noetherianity of the Balmer spectrum}\label{sec:Noetherianity}

In order to complete the proof of stratification for the big derived category of permutation modules $\T(G;R)$, where $G$ is a finite group and $R$ is a commutative Noetherian ring, we require some topological information about the Balmer spectrum of the small derived category of permutation modules $\K(G;R)$. More precisely, this space must at least be weakly Noetherian. In this section, we prove the stronger statement that $\mathrm{Spc}(\K(G;R))$ is in fact Noetherian. We also show that the base change functors to residue fields induce a set partition of the spectrum. 

When $R$ is a field, Noetherianity was established in \cite{BG25} using the conservativity of certain functors involving the so-called modular fixed points. Extending these modular fixed point functors to a more general setting appears difficult, since, as the terminology suggests, the characteristic of the field plays a crucial role in their construction. Nevertheless, an analogous construction is available for rings of prime characteristic.

Our strategy to show Noetherianity is as follows. First, we prove Noetherianity for $p$-groups, we then apply standard techniques from tt-geometry to deduce the general case. In the case of $p$-groups, we isolate two key situations: one in which the ring has positive characteristic, and another in which $p$ is invertible in the ring. Combining these two cases will yield the general result. 
%--------------------------------

\subsection{Reduction to $p$-groups}
 Recall that $R$ denotes a commutative Noetherian  ring and $\K(G;R)$ denotes the small derived category of permutation $RG$--modules. The following lemma allows us to focus on $p$-groups to prove Noetherianity of the Balmer spectrum of the small derived category of permutation modules.

\begin{lemma}\label{cover by p-subgroups}
  Let $G$ be a finite group and  $\mathcal{E}$ denote the family of $p$-power order subgroups of $G$, where $p$ runs over all  primes dividing the order of $G$. Then the restriction functors induce a continuous surjection $$\coprod_{H\in \mathcal{E}}\mathrm{Spc}(\K(H;R))\to \mathrm{Spc}(\K(G;R)).$$
\end{lemma}

We need some preparation. 

\begin{recollection}\label{rec-separable-algebra-AH}
Let $H$ be a subgroup of $G$. It has been shown in \cite{Bal17} that the adjunction given by restriction-induction $\mathrm{res}_H^G\dashv \mathrm{ind}_H^G$ on the abelian category of $RG$--modules is monadic. This monadic adjunction extends verbatim to a monadic adjunction 
\[
\mathrm{res}_H^G\colon \K(G;R)\rightleftarrows \K(H;R) \colon \mathrm{ind}_H^G
\]
and the monad can be identified with the monad obtained by tensoring with the commutative algebra $(A_H,\mu,\eta)$ with $A_H=R(G/H)$, the multiplication is given by $\mu(\gamma\otimes \gamma')=\gamma $ if $\gamma=\gamma'$ in $G/H$ or $0$ otherwise,  and  unit $\eta(1)=\sum{\gamma\in G/H} \gamma$. Moreover, the algebra $A_H$ is separable\footnote{These objects are also known as tt-rings.} and has finite degree. Indeed, the section is given by the map $\gamma \mapsto \gamma\otimes \gamma$. In particular, one can identify the category $\mathrm{Mod}_{\K(G;R)}(A_H)$ with $\K(H;R)$ and we can identify the restriction functor with extension of scalars and induction with the forgetful functor. 
\end{recollection}

\begin{proposition}\label{prop-sep-A}
    The commutative algebra  $A_\mathcal{E}=\prod_{H\in \mathcal{E}}A_H$ is separable in $\K(G;R)$ and has finite degree. Moreover, $A_\mathcal{E}$ admits descent, that is, the smallest thick tensor ideal in $\K(G;R)$ containing $A_\mathcal{E}$ is all $\K(G;R)$. 
\end{proposition}

\begin{proof}
    The first part follows since each $A_H$ is a finite degree separable commutative algebra and $\mathcal{E}$ is finite.
    The second part is a consequence of Corollary 2.2 in \cite{Car00} and its proof,  which gives us that the trivial module $R$ is a retract of $\oplus_{S\in \mathrm{Syl}(G)} R(G/S)$ where $\mathrm{Syl}(G)$ denotes the collection of Sylow subgroups of $G$.  This means that  $R$ is in the thick subcategory of $\K(G;R)$ generated by modules induced from $p$-power order subgroups of $G$, and therefore in the thick subcategory generated by $A_\mathcal{E}$. This implies the result.  
\end{proof}

\begin{remark}
One may wonder whether the previous statement can be further reduced to elementary abelian subgroups, as is often possible in other categories associated with the group algebra (for example, the stable module category). However, this is not the case, even over a field.

Indeed, consider the separable algebra $A$ constructed above but using only elementary abelian subgroups. If $A$ satisfied descent, then one could readily deduce that the Balmer spectra of the elementary abelian subgroups cover the Balmer spectrum of the whole group. However, this is far from true. For example, let $C_n$ denote the cyclic group of order $p^n$ for $n>0$, and let $R$ be a field of characteristic $p$. In \cite[Section 8]{BG25}, the authors showed that $\mathrm{Spc}(\K(C_n;R))$ has $2n+1$ points. Consequently, $\mathrm{Spc}(\K(C_1;R))$ is insufficient to cover $\mathrm{Spc}(\K(C_n;R))$ whenever $n>1$. We thank Martin Gallauer for pointing out this argument.
\end{remark}

%%%%%%%%%%%%%%%5555-----------------------------

\begin{proof}[Proof of Lemma \ref{cover by p-subgroups}]
    By the previous result we have that $A_\mathcal{E}$ satisfies the hypothesis of \cite[Theorem 3.19]{Bal16}, hence  we obtain a coequalizer of topological spaces 
    \[
    \mathrm{Spc}(\Mod_{\K(G;R)}(A_\mathcal{E}\otimes A_\mathcal{E}))\rightrightarrows \mathrm{Spc}(\Mod_{\K(G;R)}(A_\mathcal{E}))\xrightarrow{\varphi} \mathrm{Spc}(\K(G;R))
    \]
    where $\varphi$ is the induced map in spectra by extension of scalars $F_{A_\mathcal{E}}\colon\K(G;R)\to \Mod_{\K(G;R)}(A_\mathcal{E})$. But note that we can identify $\Mod_{\K(G;R)}(A_\mathcal{E})$ with $$\prod_{H\in\mathcal{E}} \Mod_{\K(G;R)}(A_H)\cong \prod_{H\in\mathcal{E}} \K(H;R)$$ and under this identification the functor        $F_{A_\mathcal{E}}$ corresponds to the functor induced by the restrictions $\K(G;R)\to \K(H;R)$, hence the result follows. 
\end{proof}

\subsection{Reduction to rings of positive characteristic}  Let $G$ denote a $p$-group, and $R$ be an arbitrary commutative Noetherian ring.

\begin{proposition}\label{prop-p-invertible}
    If $|G|$ is invertible in $R$, then we have that Balmer's comparison map \cite{Bal10} 
    \[
     \rho\colon \mathrm{Spc}(\K(G;R))\to \mathrm{Spec}^h(R)
    \]
    is a homeomorphism. 
\end{proposition}

\begin{proof}
   Using that   $|G|$ is invertible in $R$, it is straightforward to verify that any $R$--projective $RG$--module is $RG$--projective. In particular, we deduce that 
   \[
   \K(G;R)\simeq \mathbf{D}^b(\mathrm{mod}(G;R)).
   \]
   On the other hand, by Corollary \ref{spectrum of rep}, we known that 
   \[
     \rho\colon \mathrm{Spc}(\mathbf{D}^b(\mathrm{mod}(G;R)))\to \mathrm{Spec}^h(H^\ast(G;R))
     \]
     is a homeomorphism. But in this case $H^\ast(G;R)\simeq R$. This completes the proof. 
\end{proof}

With this result in place, we now show that one can reduce the Noetherianity problem to the case of rings of positive characteristic. 

\begin{proposition}\label{prop-reduction-Rcharp}
    Let $R$ be a commutative Noetherian ring. Assume that the space  $\mathrm{Spc}(\K(G;R/pR))$ is Noetherian. Then $\mathrm{Spc}(\K(G;R))$ is Noetherian as well.  
\end{proposition}

\begin{proof}
   Let $n=|G|$.  Consider the tt-functor 
   \[
   F\colon \K(G;R)\to \K(G;R/pR)\times \K(G;R[1/n])
   \]
  induced by base change along $R\to R/pR$ and $R\to R[1/n]$. We claim that this functor induces a surjection on Balmer spectra. Indeed, since any residue field of $R$ will factor either through $R/pR$ or $R[1/n]$, then we obtain a commutative diagram 
   \[
   \begin{tikzcd}
       \mathrm{Spc}(\K(G;R/pR))\amalg   \mathrm{Spc}(\K(G;R[1/n]))  \arrow[r] &   \mathrm{Spc}(\K(G))\\ 
       \coprod_{\pp\in \mathrm{Spec}(R)}\mathrm{Spc}(\K(G;k(\pp))) \arrow[ru] \arrow[u] & 
   \end{tikzcd}
   \]
   where all arrows are induced by the corresponding base change on the ground rings. Now the claim follows since the diagonal arrow is surjective by a combination of Theorem \ref{cor:joitly conservaty for T} and \cite[Theorem 1.3]{BCHS23}. The result now follows by Proposition \ref{prop-p-invertible}, which allows to  identify $\mathrm{Spc}(\K(G;R[1/n])) $ with $\mathrm{Spec}(R[1/n])$, and the latter is Noetherian since $R$ is Noetherian.  
   \end{proof}

\subsection{Case of ring of positive characteristic}
Let $G$ denote a $p$-group. In this section we assume that $R$ has characteristic $p$. 

Recall that $\mathrm{perm}(G;R)$ denotes the category of finitely generated permutation $RG$--modules. Let $N$ be a normal subgroup of $G$, and $\mathcal{F}_N=\{H\leq G\mid N\not\leq H\}$. Set  
\[
\mathrm{proj}(\mathcal{F}_N)= \mathrm{add}^\natural\{R(G/H)\mid H\in \mathcal{F}_N\}
\]
Note that the additive quotient  
\[
\mathrm{qout}^G_H\colon \mathrm{perm}(G;R)^\natural\xrightarrow{} \mathrm{perm}(G;R)^\natural/\mathrm{proj}(\mathcal{F}_N)
\]
is symmetric monoidal. Indeed, the subcategory $\mathrm{proj}(\mathcal{F}_N)$ is a tensor-ideal by the Mackey formula. Hence  the category  $\mathrm{perm}(G;R)/\mathrm{proj}(\mathcal{F}_N)$ is symmetric monoidal, and the functor $\mathrm{qout}^G_H$ is strongly monoidal. 

\begin{proposition}\label{prop-equivalence-normal}
    Let $N\unlhd G$ be a normal subgroup. Then the composite 
    \[
   \mathrm{perm}(G/N;R)\xrightarrow[]{\mathrm{Infl}} \mathrm{perm}(G;R)^\natural\xrightarrow{\mathrm{qout}} \mathrm{perm}(G;R)^\natural/\mathrm{proj}(\mathcal{F}_N)
    \]
    is an equivalence of tensor additive categories. 
\end{proposition}
\begin{proof}
    Given that we are assuming that the characteristic of $R$ is $p$, the same proof as in \cite[Proposition 5.4]{BG25} gives us the result. 
\end{proof}

\begin{construction}\label{cons-triangular fixed points}
   For a normal subgroup $N$ we obtain a tt-functor  on bounded homotopy categories 
\[
\psi^N \colon \K(G;R)\xrightarrow{} \K(G/N;R)
\]
by applying the equivalence form Proposition \ref{prop-equivalence-normal} degree-wise. Now, let $H\leq G$ be a subgroup. The $H$-modular fixed points functor is the tt-functor given as the composite 
\[
\Psi^H\colon \K(G;R) \xrightarrow[]{\mathrm{res}}\K(N_G(H);R)\xrightarrow[]{\Psi^{N_G(H),H}}\K(\WGH;R).
\]
\end{construction}

Let us record a simple observation that relates the functor $\Psi^H$ with modular fixed points introduced in \cite{BG25} after base change. But first, recall that our assumption on $R$ ensures that any residue field of it has characteristic $p$. 

\begin{proposition}
\label{prop-basechange-psi}
Let $\pp\in \mathrm{Spec}(R)$, and $H$ be a subgroup of $G$. Then we have a commutative diagram 
\[
\begin{tikzcd}
     \K(G;R) \arrow[r,"\Psi^H"] \arrow[d] & \K(\WGH;R) \arrow[d]\\
     \K(G;k(\pp))\arrow[r,"\Psi^H"] &  \K(\WGH;k(\pp))
\end{tikzcd}
\]
where the vertical arrows are induced by base change to the residue fields, and the bottom map is the  modular $H$-fixed points functor. 
\end{proposition}

\begin{proof}
    This already holds at the level of the additive categories of permutation modules and noticing that $\mathrm{proj}(\mathcal{F}_N)$ is sent to the corresponding tensor ideal defining the modular fixed points under base change. 
\end{proof}

 We need some preparations before we address  Noetherianity of the Balmer spectrum in this setting. 
 We will use the same notation as in \cite{BG22}. 

 \begin{recollection}
 Let $G$ be a $p$-group. Recall that there is a localization functor 
 \[
 \Upsilon^G\colon\T(G;R)\to \mathbf{K}^\textrm{perm}(\mathrm{Inj}(RG))
 \]
 where the right-hand side category is the localizing subcategory of $\mathbf{K}( \mathrm{Inj}(RG))$ generated by Tate resolutions of permutation $RG$--modules. This functor,  on compacts, gives a localization functor 
 \[
 \Upsilon^G\colon\K(G;R)\to \mathbf{D}^\mathrm{perm}(RG)
 \]
 where the right-hand side category denotes the thick subcategory of $\mathbf{D}^b(RG)$ generated by the finitely generated permutation $RG$-modules. In particular, when $R$ is regular, there are equivalences $\mathrm{KInj}^\textrm{perm}(G;R)\simeq \mathbf{K}( \mathrm{Inj}(RG))$ and $\mathbf{D}^\mathrm{perm}(RG)\simeq \mathbf{D}^b(RG)$, see Remark \ref{Db is a localization of KG}.      
 \end{recollection}

Now we are ready to prove Noetherianity for $p$-groups over rings of characteristic $p$.

\begin{theorem}\label{cover of spc for p-group}
  Let $G$ be a $p$-group and $R$ be a commutative Noetherian ring of characteristic $p$. Then there is a continuous surjection 
  \[
  \coprod_{H\in\mathrm{Sub} (G)} \mathrm{Spc}(\mathbf{D}^b(\mathrm{mod}(\WGH;R))) \to \mathrm{Spc}(\K(G;R))
  \]
  where $\mathrm{Sub} (G)$ denotes the family of subgroups of $G$. In particular, the Balmer spectrum of $\K(G;R)$ is a Noetherian space. 
\end{theorem}

\begin{proof}
   For each  subgroup $H\leq G$ we consider the composition 
   \begin{align*}
       \Theta_{H,R}\colon\mathrm{Spc}(\mathbf{D}^\mathrm{perm}(R\WGH))\hookrightarrow \mathrm{Spc}(\K(\WGH;R))&  \xrightarrow{{\psi}^{H}} \mathrm{Spc}(\K(G;R))
   \end{align*}
    where the right hand side map is the one induced by the functor $\Psi^H$ from Construction \ref{cons-triangular fixed points}, and the left hand side map is induced  by the  localization functor $\Upsilon^G\colon \K(G;R)\to \mathbf{D}^\mathrm{perm}(RG)$ which is an injective continuous map.

   Now, consider a prime $\pp$ in $R$. We have a commutative diagram induced by base change along $R\to k(\pp)$
   \[
    \begin{tikzcd}
  \K(G;R)  \arrow[d] \arrow[r,"\Psi^H"] & \K(\WGH, R)  \arrow[r,"\Upsilon^G"] \arrow[d] & \mathbf{D}^\mathrm{perm}(R\WGH) \arrow[d]  \\
  \K(G;k(\pp)) \arrow[r,"\Psi^H"] &  \K(\WGH, k(\pp)) \arrow[r,"\Upsilon^G"] & \mathbf{D}^\mathrm{perm}(k(\pp)\WGH)
    \end{tikzcd}
   \]
 which give us a commutative diagram of topological spaces 
\[
\begin{tikzcd}
           \displaystyle   \coprod_{H\in \mathrm{Sub}(G)} \mathrm{Spc}(\mathbf{D}^{\mathrm{perm}}(R\WGH)) \arrow[rr,"\coprod\Theta_{H,R}^\ast"] & & \mathrm{Spc}(\K(G;R)) \\
    \displaystyle  \coprod_{H\in \mathrm{Sub}(G)}  \coprod_{\pp\in \mathrm{Spec}(R)} \mathrm{Spc}(\mathbf{D}^{\mathrm{perm}} (k(\pp)\WGH)) \arrow[rr,"\coprod\Theta_{H;k(\pp)}^\ast"] \arrow[u]  & & \displaystyle \coprod_{\pp\in \mathrm{Spec}(R)} \mathrm{Spc}(\K(G;k(\pp)) \arrow[u]
    \end{tikzcd}
\] The right-vertical map is surjective by Theorem \ref{cor:joitly conservaty for T} combined with \cite[Theorem 1.3]{BCHS23}. We claim that the bottom map is surjective as well. Note that this would complete the proof. Since  $k(\pp)$ has characteristic $p$, then the map 
\[
\displaystyle \coprod_{H\in \mathrm{Sub}(G)}\mathrm{Spc}(\mathbf{D}^\mathrm{perm}(k(\pp)))\xrightarrow{\coprod\Theta_{H;k(\pp)}^\ast} \mathrm{Spc}(\K(G;k(\pp)))
\]
agrees with the map on spectra considered in  \cite[Corollary 7.2]{BG25} which is surjective. Indeed, in this case the category  $\mathbf{D}^\mathrm{perm}(k(\pp)\WGH)$ agrees with $\mathbf{D}^b(k(\pp)\WGH)$. It follows that  $\coprod\Theta_{H;k(\pp)}^\ast$ is surjective. Therefore, we deduce that the top map from the previous square is surjective as we wanted.

On the other hand, the inclusion $\mathbf{D}^\mathrm{perm}(RG)\subseteq \mathbf{D}^b(\mathrm{mod}(G;R))$ induces a continuous surjection $\mathrm{Spc}(\mathbf{D}^b(\mathrm{mod}(G;R)))\to \mathrm{Spc}(\mathbf{D}^\mathrm{perm}(RG))$. Thus we obtain a surjective map 
\[
\coprod_{H\leq G} \mathrm{Spc}(\mathbf{D}^b(\mathrm{mod}(\WGH;R))) \to \coprod_{H\leq G} \mathrm{Spc}(\mathbf{D}^\mathrm{perm}(R\WGH))\to \mathrm{Spc}(\K(G;R)).
\]
 As an immediate consequence, we get that $\mathrm{Spc}(\K(G;R))$ is Noetherian since the spaces $\mathrm{Spc}(\mathbf{D}^b(\mathrm{mod}(\WGH;R)))$ are Noetherian; see Corollary~\ref{spectrum of rep}. 
\end{proof}

\subsection{The general case} We have now all the ingredients to prove Noetherianity of the spectrum: 

\begin{corollary}\label{cover of spc}
    Let $G$ be a finite group and $R$ be a commutative Noetherian ring.  Then  $\Spc(\K(G;R))$ is a Noetherian space.
\end{corollary}

\begin{proof}
  By Lemma \ref{cover by p-subgroups}, it suffices to prove the claim for $p$-groups. For a $p$-group, the claim then follows from Proposition \ref{prop-reduction-Rcharp}, where the required hypotheses are verified in Proposition \ref{prop-p-invertible} and Theorem \ref{cover of spc for p-group}.
\end{proof}

\subsection{A set partition of the spectrum}  The goal of this section is to show that the base change functors $\iota_\pp^\ast$ induce a partition of the spectrum of $\K(G;R)$. We begin with the case where $G$ is a $p$-group. 

\begin{proposition}\label{prop-setpartition}
    Let $G$ be a finite $p$-group and $R$ be a commutative Noetherian ring. For $\pp\in \Spec(R)$, write $\iota_\pp^\ast$ to denote the functor induced by base change along $R\to k(\pp)$. Then 
    \[
   \coprod\mathrm{Spc}(\iota_\pp^\ast) \colon \coprod_{\pp\in \mathrm{Spec}(R)} \mathrm{Spc}(\K(G;k(\pp))\to \mathrm{Spc}(\K(G;R))
    \]
    is bijective. 
\end{proposition}

 We need some preparation. For now, assume that $G$ is a $p$-group.

\begin{notation}
  Let $H\leq G$ be a subgroup. We write  $\rho^G_H$ to denote $\Spc(\res^G_H)$. Let $\pp\in \Spec(R)$. We write $\check\iota_\pp^\ast$ to denote $\mathrm{Spc}(\iota_\pp^\ast)$. If the characteristic of $k(\pp)$ is $p$, then we write $\check\psi^H_{k(\pp)}$ as shorthand for  
  \[
  \Spc(\Upsilon^{\WGH}\circ\Psi^{\WGH})\colon \Spc(\mathbf{D}^b(k(\pp)\WGH)) \to \mathrm{Spc}(\K(G;k(\pp))).
  \]
\end{notation}

\begin{notation}
 Given $\pp\in \Spec(R)$ with $\mathrm{char}(k(\pp))=p$, $H\leq G$, and $\P\in V_{\WGH;k(\pp)}:=\Spc(\mathbf{D}^b(k(\pp)\WGH))$, we set
 \[
 \P_{G,R}(H,\P,\pp):= \check\iota_\pp^\ast(\check\psi^H_{k(\pp)}(\P))\in \Spc(\K(G;R)).
 \]
\end{notation}

\begin{remark}\label{rem-partition-fieldcase}
  Let $k$ be a field of positive characteristic $p$, and let $G$ be a $p$-group. By \cite[Proposition 7.32]{BG25} we have that the modular fixed points functors induce a set partition: 
  \[
  \coprod \check{\psi}^H_k\colon \coprod_{H\in \mathrm{Sub}(G)/G} V_{\WGH;k}\xrightarrow[]{\simeq} \Spc(\K(G;k)).
  \] 
  In particular, the image of a point $\P\in V_{\WGH;k}$ under the map from above is denoted by $\P_{G;k}(H;k)$ which is the notation from \cite{BG25}. 
\end{remark}

\begin{remark}\label{rem-psi-basechange}
  For any residue field $k(\pp)$ of characteristic $p$, note that   $\check\iota_\pp^\ast \P_{G;k}(H,\P)$ corresponds to $\P_{G,R}(H,\P,\pp)$. 
\end{remark}

\begin{proof}[Proof of Proposition \ref{prop-setpartition}]
    Since the family of functors $\iota_\pp^\ast$ indexed over all primes of $R$ is jointly conservative, we know that $\coprod\check\iota^\ast_\pp$ is surjective. Moreover, by the naturality of Balmer's comparison map, it is enough to verify that, for each prime $\pp$ of $R$, the map $\check\iota^\ast_\pp$ is injective. Note that if the residue field $k(\pp)$ has characteristic coprime to $p$, then $\check\iota^\ast_\pp$ is trivially injective since $\Spc(\K(G;k(\pp)))$ is a point. Hence, assume otherwise, and write $k$ to denote $k(\pp)$.

    In view of Remark \ref{rem-partition-fieldcase}, it is enough to check that  $\check\iota_\pp^\ast\check\psi^H_k$ is injective for all $H\leq G$. The case of $H=G,1$ are straightforward. Indeed, 
    since the domain of $\check\psi^G_k$ is a point, this is trivial. Now, we have that $\check\psi^1_k=\Spc(\Upsilon^G_{k})\colon V_{G;k}\to \Spc(\K(G;k))$, and $\check\iota_\pp^\ast\check\psi^1_k$ agrees with the map 
    \[
    V_{G;k}\xrightarrow[]{} V_{G,R} \xrightarrow[]{\Upsilon^G}  \Spc(\K(G;R))
    \]
    where the first map is induced by base change between categories of representations. Both of these maps are injective, hence $\check\psi^1_k$ is injective as we wanted. 

     We will show that $\check\iota_\pp^\ast$ is injective  by induction on the cardinality of $G$ which is $p^n$. The base case is $n=1$. This follows by the previous discussion. Assume then that $n>1$, and let $H\leq G$ be a proper subgroup. We are left to verify that $\check\iota_\pp^\ast\check\psi^H_k$ is injective. Let  $\P$ and $\P'$ be in $V_{\WGH;k}$. If $\P_{G,R}(H,\P,\pp)=\P_{G,R}(H,\P',\pp)$, then unpacking the definitions it is easy to see that 
     \begin{equation}\label{eq-rho1}
         \rho^G_N (\P_{N,R}(H,\P,\pp))= \rho^G_N (\P_{N,R}(H,\P',\pp))
     \end{equation}
     where $N$ denotes $N_G(H)$. Indeed, one only needs to verify that the functor $\iota_\pp^\ast$ is compatible with the restriction functor $\mathrm{res}^G_N$, which is straightforward.  Now, by Recollection \ref{rec-separable-algebra-AH} and \cite[Theorem 3.19]{Bal16} applied to the separable commutative algebra $A_N$, we obtain a diagram
     \[
      \begin{tikzcd}
         \displaystyle \coprod_{[g]\in H\backslash G/H} \Spc(\K(N\cap{}^gN;k))\arrow[r,shift left=1] \arrow[r,shift left=-1] \arrow[d] & \Spc(\K(N;k)) \arrow[d] \arrow[r,"\rho^G_H"] & \mathrm{supp}(k(G/N)) \arrow[d] \\ 
         \displaystyle   \coprod_{[g]\in H\backslash G/H} \Spc(\K(N\cap{}^gN;R))\arrow[r,shift left=1] \arrow[r,shift left=-1] & \Spc(\K(N;R)) \arrow[r,"\rho^G_H"] & \mathrm{supp}(k(G/N))
      \end{tikzcd}
     \]
     where the horizontal diagrams are coequalizers of topological spaces (see see \cite[Proposition 4.7]{BG25}), and the vertical maps are induced by $\iota^\ast_\pp$. In fact, this diagram is commutative using again that restriction functor commutes with base change.  

     By Equation \ref{eq-rho1}, we deduce that there exists $g\in G$ and  $\mathcal{Q}\in \Spc(\K(N\cap {}^gN;R))$ such that 
     \begin{equation}
         \label{eq-rho2}
          \rho^N_{N\cap {}^gN}\mathcal Q = \P_{N,R}(H,\P,\pp) \; \mbox{and} \; (\rho^{{}^gN}_{N\cap {}^gN}\mathcal Q)^g = \P_{N,R}(H,\P',\pp).
     \end{equation}
    By the induction hypothesis and Remark \ref{rem-partition-fieldcase} we deduce that  $\mathcal{Q}$ is necessarily  of the form 
     \[
     \P_{N\cap ^{g}N,R}(L,\P'',\pp)=\check\iota_\pp^\ast \P_{N\cap ^{g}N,k}(L,\P'')
     \]
     for some $L\leq N\cap {}^gN$, and some $\P''\in V_{\Weyl{(N\cap {}^gN)}{L},k}$ (see Remark \ref{rem-psi-basechange}). Now, by the commutativity of the above diagram, the injectivity of the middle vertical map (by the induction hypothesis) and Equation \ref{eq-rho2}, we deduce that 
     \[
     \rho^N_{N\cap {}^gN}\mathcal \P_{N\cap ^{g}N,k}(L,\P'') = \P_{N,k}(H,\P) \; \mbox{ and } \; (\rho^{{}^gN}_{N\cap {}^gN}\mathcal \P_{N\cap ^{g}N,k}(L,\P''))^g = \P_{N,k}(H,\P').
     \]
     We are in the same situation as in the proof of \cite[Proposition 7.14]{BG25} which gives us that $\P=\P'$, as we wanted. This concludes the proof. 
\end{proof}

\begin{theorem}\label{thm-setpartition-general}
 Let $G$ be a finite group and $R$ be a commutative Noetherian ring. Then 
    \[
   \coprod\mathrm{Spc}(\iota_\pp^\ast) \colon \coprod_{\pp\in \mathrm{Spec}(R)} \mathrm{Spc}(\K(G;k(\pp))\to \mathrm{Spc}(\K(G;R))
    \]
    is bijective.    
\end{theorem}
\begin{proof}
  Again, surjectivity follows from the conservativity of the functors $\iota_\pp^\ast$. By the naturality of Balmer's comparison map, it suffices to verify that each component of the above map is injective.

  Recall that $\mathcal{E}$ denotes the family of subgroups of $G$ of $p$-power order, where $p$ ranges over the prime divisors of $|G|$. Fix $\pp\in \mathrm{Spec}(R)$ such that $\mathrm{char}(k(\pp))$ divides $|G|$, since otherwise the claim is trivial. Write $k:=k(\pp)$. Applying \cite[Theorem 3.19]{Bal16} to Proposition \ref{prop-sep-A} yields the following diagram
   \[
      \begin{tikzcd}
         \displaystyle \coprod_{\substack{H,K\in \mathcal{E}\\ [g]\in H\backslash G/K}} \Spc(\K(H\cap{}^gK;k))\arrow[r,shift left=1] \arrow[r,shift left=-1] \arrow[d] & \displaystyle \coprod_{H\in \mathcal{E}}\Spc(\K(H;k)) \arrow[d] \arrow[r,""] & \Spc(\K(G;k)) \arrow[d] \\ 
         \displaystyle  \coprod_{\substack{H,K\in \mathcal{E}\\ [g]\in H\backslash G/K}} \Spc(\K(H\cap{}^gK;R)) \arrow[r,shift left=1] \arrow[r,shift left=-1] & \displaystyle \coprod_{H\in \mathcal{E}}\Spc(\K(H;R)) \arrow[r,""] & \Spc(\K(G;R))
      \end{tikzcd}
     \]
    in which the horizontal arrows are coequalizers in the category of topological spaces, and the vertical maps are induced by $\iota_\pp^\ast$. The diagram commutes by compatibility of restriction with base change. A straightforward diagram chase, together with Proposition \ref{prop-setpartition}, then shows that the rightmost horizontal map is injective.
\end{proof}

\begin{remark}
In \cite{DG25}, the author, in collaboration with U.~Dubey, addresses a set-theoretic description of the Balmer spectrum for general finite groups. Although we could simply refer to that work, we have chosen to include an alternative proof here in order to keep the present results self-contained and independent.
\end{remark}

\section{Beyond big tt-categories}\label{sec: beyond big}

In this section, we discuss an application of Lemma~\ref{coro 2} to tt-categories with small coproducts that are not necessarily big; that is, tt-categories that need not be rigidly compactly generated. Our main source of examples comes from the author’s work in \cite{Gom24b}. For the reader’s convenience, we briefly recall the relevant definitions.

\begin{recollection}
    Let $k$ be a field and $G$ be a group with a finite-dimensional model for the classifying space for proper actions (e.g. a finite group, a discrete $p$-toral group, etc.). The stable module $\infty$-category $\mathrm{StMod}(kG)$ is defined as the underlying $\infty$-category of the Quillen model structure on the category of $kG$--modules $\mathrm{Mod}(kG)$ given as follows: 
    \begin{itemize}
    \item[(i)] A fibration is a surjective map.
    \item[(ii)] A cofibration  is an injective map with Gorenstein projective  cokernel.  
    \item[(iii)] A weak equivalence  is a stable isomorphism, that is,  it is an isomorphism in $\underline{\StMod}(kG)$. 
\end{itemize}
 Here, $\underline{\StMod}(kG)$ denotes  the category whose objects are $kG$--modules and whose morphisms are given by complete cohomology \cite[Definition 1.4]{Gom24b}.  In particular, $\StMod(kG)$ is a presentable, symmetric monoidal, stable $\infty$-category where the tensor product commutes with colimits in each variable separately \cite[Proposition 2.4]{Gom24b}. The monoidal structure is inherited from the one on $\Mod(kG)$ given by the tensor product over $k$ endowed with the diagonal action of $G$. Note that the monoidal unit is $k$ with the trivial action of $G$.  See \cite[Section 2]{Gom24b} and references therein for further details.
\end{recollection}

\begin{proposition}\label{generation of stab}
    Let $k$ be a field of characteristic $p$, and $G$ be a group with a finite-dimensional model 
$X$ for the classifying space for proper actions. Consider the collection $\mathcal{E}$ of elementary abelian finite $p$-subgroups of $G$, along with the family of restriction functors 
\[
\{\mathrm{res}_E\colon \StMod(kG)\to \StMod(kE)\}_{E\in \mathcal{E}}.
\]
Then  $k\in \mathrm{Loc}^\otimes(\coprod_{E\in \mathcal{E}}\mathrm{ind}_E (k))$ where $\mathrm{ind}_E$ denotes the induction functor which is the left adjoint of $\mathrm{res}_E$.  In particular, Building holds with respect to the family $\{\mathrm{res}_E\}_{E\in \mathcal{E}}$. 
\end{proposition}

\begin{proof}
    By the hypothesis on $G$, we know that the augmented cellular chain complex $C_\ast(X;k)$ of $X$ with coefficients in $k$ is an exact complex of $kG$--modules   
    \[
    0\to C_n\to...\to C_0\to k\to 0
    \]
    such that each $C_i$ is a direct sum  of modules of the form $\mathrm{ind}_H(k)$, for $H$ a finite subgroup of $G$. This immediately give us that 
    \[
    k\in \mathrm{Loc}^\otimes\left(\coprod_{H\in \mathcal{F}}\mathrm{ind}_H (k)\right)
    \]
    where $\mathcal{F}$ denotes the family of finite subgroups of $G$. We can change the family $\mathcal{F}$ by $\mathcal{E}$ using that $\{\mathrm{ind}_E(k)\mid E\leq H\}$ generates $k$ in $\mathrm{StMod}(kH)$ for any $H\in \mathcal{F}$, and the fact that $\mathrm{ind}_H^G\circ \mathrm{ind}_E^H\simeq \mathrm{ind}_E^G$. 

    On the other hand, it is well known that $\mathrm{ind}\dashv \mathrm{res}$ satisfy the projection formula. Thus the second claim is a consequence of Corollary~\ref{coro 2}.
\end{proof}

\begin{remark}
The stable module $\infty$-category for a finite group is well known to be rigidly-compactly generated. However, the same does not necessarily hold for infinite groups. One can deduce from the conservativity of the family of restriction functors along the family of finite subgroups that the stable category is compactly generated, however, the monoidal unit fails to be compact (see \cite[Example 2.3, Remark 2.7]{Gom24a}). Therefore, our approach indeed extends beyond big tt-categories. 
\end{remark}

\begin{remark}
It is straightforward to verify that the family $\mathcal{E}$ from Proposition \ref{generation of stab} can be replaced by a set of representatives, up to conjugacy, of elementary abelian subgroups, or by any family of subgroups containing them. On the other hand, the previous proposition provides a detection result in terms of restrictions to elementary abelian subgroups. However, this is a well-known consequence of the existence of $C_\ast(X)$ and Chouinard's theorem for finite groups.  
\end{remark}

%------------ bibliography

\bibliographystyle{alpha}
\bibliography{mybibfile} 
%--------------------------------------------------------------------------------------%

\end{document}